\documentclass[12pt]{article}

\usepackage{ctex}
\usepackage{amsmath}
\usepackage{amsfonts,amssymb}
\usepackage{extarrows}
\usepackage{geometry}
\usepackage[hidelinks]{hyperref}

\usepackage[utf8]{inputenc}
\def\q{\hfill\rule{1ex}{1ex}}
\def\0{\emptyset}

\def\q{\hfill\rule{1ex}{1ex}}

\newtheorem{theorem}{Theorem}[section]
\newtheorem{definition}[theorem]{Definition}
\newtheorem{lemma}[theorem]{Lemma}
\newtheorem{claim}[theorem]{Claim}
\newtheorem{observation}[theorem]{Observation}
\newtheorem{cor}[theorem]{Corollary}

\newtheorem{remark}[theorem]{Remark}
\newtheorem{conjecture}[theorem]{Conjecture}

\def\q{\hfill\rule{1ex}{1ex}}
\newenvironment{proof}{{\noindent\it Proof.}}{\hfill $\square$\par}

\newcommand\invt{\mathrm{inv}(T)}
\newcommand\invd{\mathrm{inv}(D)}
\newcommand\invdi{\mathrm{inv}(D_i)}
\newcommand\invn{\mathrm{inv}(n)}
\newcommand\fk{\mathbb{F}_2^k}

\usepackage{graphicx}
\renewcommand{\figurename}{Figure}

\usepackage{enumerate}

\usepackage{
	amsmath,			
	amssymb,			
	enumerate,		    
	graphicx,			
	lastpage,			
	multicol,			
	multirow,			
	pifont,			    
}

\usepackage[numbers]{natbib}
\renewcommand\refname{References}
\renewcommand\abstractname{Abstract}

\begin{document}


\title{The inversion number of dijoins and blow-up digraphs}
\author{
    {\small\bf Haozhe Wang}\thanks{ Department of Mathematical Sciences, Tsinghua University, Beijing 100084, China. email:  whz22@mails.tsinghua.edu.cn}\quad
    {\small\bf Yuxuan Yang}\thanks{ Correspondence Author. School of Science, Beijing University of Posts and Telecommunications, Beijing 100876, China. email:  yangyx@bupt.edu.cn}\quad
    {\small\bf Mei Lu}\thanks{ Department of Mathematical Sciences, Tsinghua University, Beijing 100084, China. email: lumei@tsinghua.edu.cn
}\\
}

\date{}

\maketitle\baselineskip 16.3pt

\begin{abstract}
	For an oriented graph \(D\), the \emph{inversion} of \(X \subseteq V(D)\) in \(D\) is the digraph obtained
from \(D\) by reversing the direction of all arcs with both ends in \(X\). The \emph{inversion number} of
\(D\), denoted by \(\invd\), is the minimum number of inversions needed to transform \(D\) into an acyclic
digraph. In this paper, we first show that \( \mathrm{inv} (\overrightarrow{C_3} \Rightarrow D)= \invd +1 \) for any
oriented graph \(D\) with even inversion number \(\invd\), where the dijoin \(\overrightarrow{C_3} \Rightarrow D\) is the oriented
graph obtained from the disjoint union of \(\overrightarrow{C_3}\) and \(D\) by adding all arcs from \(\overrightarrow{C_3}\) to \(D\).
Thus we disprove the conjecture of Aubian el at. \cite{2212.09188} and the conjecture of Alon el at. \cite{2212.11969}. We also study the blow-up graph
which is an oriented graph obtained from a tournament by replacing all vertices into oriented graphs. We
construct a tournament \(T\) with order \(n\) and \(\invt=\frac{n}{3}+1\) using blow-up graphs.
\end{abstract}


\vskip.3cm


\section{Introduction}

In this paper we only consider digraphs without loops, parallel edges and \(2\)-cycles. In the following content, we will use digraph and oriented graph interchangeably. We denote by \(V(D)\) and \(A(D)\) the vertex set and the arc set of a digraph \(D\), respectively. The subdigraph of \(D\) induced by a subset \(X \subseteq V(D)\) is denoted by \(D \left\langle X \right\rangle\).
The \emph{inversion} of \(X\) in \(D\) is the digraph obtained from \(D\) by reversing all arcs in \(A(D \left\langle X \right\rangle)\), which we denote by Inv\((D;X)\). If \((X_i)_{i \in I}\) is a family of subsets of \(V(D)\), then Inv\((D;(X_i)_{i \in I})\) is the digraph obtained from \(D\) by inverting \(X_i\) one after another. Note that the ordering of the inversions doesn't matter. We say \((X_i)_{i \in I}\) is a \emph{decycling family} if Inv\((D;(X_i)_{i \in I})\) is acyclic. Furthermore,
if \(|I|=k\), we say \((X_i)_{i \in I}\) is a \(k\)-decycling family.
The \emph{inversion number} of \(D\) is the minimum cardinality of a decycling family, which we denote by inv\((D)\). For a positive integer \(n\), denote \(\invn= \mathrm{max}\{\invd | \text{\(D\) oriented graph of order \(n\)}\}\).

The inversion problem was first introduced by Belkhechine et al. (\cite{belkhechine:tel-00609544} \cite{BELKHECHINE2010703}). They mainly concerned about \(\invn\). Recently, Bang-Jensen et al. \cite{2105.04137}, Alon el at. \cite{2212.11969} and Aubian el at. \cite{2212.09188} have obtained many results regarding all aspects of the problem.

It is natural to consider the inversion number under some graph operations. For digraphs $L$ and $R$, the \emph{dijoin} $L\Rightarrow R$ from $L$ to $R$ is the digraph composed by disjoint union of $L$ and $R$, with an edge $uv$ for all $u\in V(L)$ and $v\in V(R)$.
Bang-Jensen et al. \cite{2105.04137} showed that, for two strongly connected oriented graphs
\(L\) and \(R\) such that \(\mathrm{inv}(L),\mathrm{inv}(R) \ge 2\),  \(\mathrm{inv}(L \Rightarrow R) \ge 4\). They also gave the following conjecture.
\begin{conjecture}[Bang-Jensen et al. \cite{2105.04137}]\label{dijoinconjecture}
    For oriented graphs \(L\) and \(R\), we have \(\mathrm{inv}(L \Rightarrow R) =\mathrm{inv}(L)+\mathrm{inv}(R)\).
\end{conjecture}

This is called ``dijoin conjecture", which is an important cornerstone of inversion number problem. Conjecture \ref{dijoinconjecture} is trivial when one of $\mathrm{inv}(L)$ and $\mathrm{inv}(R)$ is zero. Bang-Jensen et al. \cite{2105.04137} showed it is true when $\mathrm{inv}(L)+\mathrm{inv}(R)\le 3$, and Alon el at. \cite{2212.11969} showed it is true when $\mathrm{inv}(L)=\mathrm{inv}(R)=2$. However, Alon el at. \cite{2212.11969} and Aubian el at. \cite{2212.09188} independently found Conjecture \ref{dijoinconjecture} is not correct in general.

Note that the dijoin operation is not a symmetric operation, which means $L\Rightarrow R$ and $R\Rightarrow L$ are different in general. However, Conjecture \ref{dijoinconjecture} is symmetric in some sense, since the graphs obtained by inverting the whole vertex set give the corresponding examples in the other direction. Actually, we have $\mathrm{inv}(D)=\mathrm{inv}(\mathrm{Inv}(D;V(D)))$.

To disprove Conjecture \ref{dijoinconjecture}, Aubian el at. \cite{2212.09188} gave an explicit construction.

\begin{theorem}[Aubian el at. \cite{2212.09188}]\label{thm1.2}
	For every odd integer \(k \ge 3\), there is a tournament \(T_k\) with \(\mathrm{inv}(T_k)=k\) such that
    \(\mathrm{inv}(T_k \Rightarrow R) \le k+ \mathrm{inv}(R) -1\) for every oriented graph \(R\) with \(\mathrm{inv}(R) \ge 1\).
\end{theorem}

They also conjectured that the same statement also holds for every even integer $k\ge 4$.

\begin{conjecture}[Aubian el at. \cite{2212.09188}]\label{conjecture}
    For any \(k \ge 3\), there is a tournament \(T_k\) with \(\mathrm{inv}(T_k)=k\) such that
    \(\mathrm{inv}(T_k \Rightarrow R) < k+\mathrm{inv}(R)\) for all \(R\) with \(\mathrm{inv}(R) \ge 1\).
\end{conjecture}

Independently and almost simutaneously, Alon el at. \cite{2212.11969} exhibited a  tournament \(R\) with 9 vertices such that \(\mathrm{inv}(R)=\mathrm{inv}(\overrightarrow{C_3} \Rightarrow R)=3\), where \(\overrightarrow{C_3}\) is the directed cycle on three vertices. They also give the following conjecture, which is similar with Conjecture \ref{conjecture}.

\begin{conjecture}[Alon el at. \cite{2212.11969}]\label{conjecturealon}
    For all $l,r\in \mathbb{N}$ with \(l \ge 3\) or $r\ge 3$, there exists oriented graphs \(L\) and $R$ with \(\mathrm{inv}(L)=l\) and  \(\mathrm{inv}(R)=r\), but  \(\mathrm{inv}(L\Rightarrow R)<l+r\).
\end{conjecture}

In this paper, we prove the following theorem and then disprove both Conjecture \ref{conjecture} and Conjecture \ref{conjecturealon}. It is quite surprising and tells us that  Conjecture \ref{dijoinconjecture} is still correct in some cases.

\vspace{.2cm}
\begin{theorem}\label{mainresult}
	Let \(k \ge 2\) be an even integer. For all oriented graphs \(L\) and $R$ with \(\mathrm{inv}(L)=1\) and  \(\mathrm{inv}(R)=k\), we have
 \(\mathrm{inv}(L \Rightarrow R)=\mathrm{inv}(R \Rightarrow L)=\mathrm{inv}(L)+\mathrm{inv}(R)=1+k\).
\end{theorem}

In conclusion, here are the current results for ``dijoin conjecture". For oriented graphs \(L\) and \(R\), if $\mathrm{inv}(L)=l$ and $\mathrm{inv}(R)=r$, we have $\mathrm{inv}(L \Rightarrow R) =l+r$ for
\begin{equation*}
    (l,r)\in \{(0,0),(0,k),(k,0),(1,1),(1,2k),(2k,1),(2,2)\},
\end{equation*}
where $k$ is a positive integer. Also, there are counterexamples with $\mathrm{inv}(L \Rightarrow R) <l+r$ when either $l$ or $r$ is an odd number at least 3 and $lr\neq 0$. The remaining cases are still open, which are the cases when $l$ and $r$ are both positive even numbers except for $l=r=2$.

On the other hand, a natural generalization of dijoin operation is so-called $k$-join, which is a graph obtained by dijoin $k$ digraphs one after another. Additionally, here is a further generalization.
\begin{definition}[blow-up graph]
    Let \(H\) be a digraph with \(|V(H)|=n\). Label the vertices in \(H\) as \(v_1,v_2,\ldots,v_n\).
    For pairwise vertex disjoint digraphs \(D_1,\ldots, D_n\), define the \emph{blow-up} graph \(H[D_1,D_2, \ldots,D_n]\) is the digraph whose vertex set and arc set are
\begin{equation*}	
\begin{aligned}
	V(H[D_1,D_2, \cdots,D_n])&=\bigcup_{i=1}^n V(D_i),\\
 A(H[D_1,D_2, \cdots,D_n])&=\bigcup_{v_iv_j \in A(H)}\{xy| x \in V(D_i) , y \in V(D_j)\}
     \cup \bigcup_{k=1}^n A(D_k).
\end{aligned}
\end{equation*}
\end{definition}
Moreover,  if \(D_i=D\) are all the same, we simply denote \(H[D_1, D_2, \ldots , D_n]\) by \(H[D]_n\). If \(V(H)=\{x,y\}\) and \(A(H)=\{xy\}\), it gives dijoin \(H[D_1,D_2]=D_1 \Rightarrow D_2\). Also, $k$-join comes from \(H=TT_k\), which is a tournament with order \(k\) and arc set \(A(TT_k)=\{v_iv_j|i<j\}\), then we denote \(TT_k[D_1, D_2, \ldots , D_k]\) by \([D_1, D_2, \ldots , D_k]\) and denote \(TT_k[D]_k\) by \([D]_k\).

To prove Theorem \ref{mainresult}, we start with the case when $L=\overrightarrow{C_3}$.

\vspace{.2cm}
\begin{theorem}\label{thm1.3}
	Let \(k \ge 2\) be an even integer and \(D\) be an oriented graph with \(\mathrm{inv}(D)=k\). Then
 \(\mathrm{inv}(\overrightarrow{C_3} \Rightarrow D)=\mathrm{inv}(\overrightarrow{C_3})+\mathrm{inv}(D)=1+k\).
\end{theorem}

The main technique to prove Theorem \ref{thm1.3} is to investigate the rank of some related matrices over $\mathbb{F}_2$. From the same idea, we have some other interesting results. We provide an equivalent condition when \(\mathrm{inv}(\overrightarrow{C_3} \Rightarrow D)= \invd =k\), and prove the following theorems.

\begin{theorem}\label{direction}
	Let \(D\) be oriented graph. Then
 \(\mathrm{inv}(\overrightarrow{C_3} \Rightarrow D)=\mathrm{inv}(D \Rightarrow \overrightarrow{C_3})\).
\end{theorem}

\begin{theorem}\label{abnormal}
    \(\mathrm{inv}([\overrightarrow{C_3} , \overrightarrow{C_3} , D])=\mathrm{inv}(\overrightarrow{C_3} \Rightarrow D)+1\) for every oriented graph \(D\).
\end{theorem}

We also have the following generalized result as a corollary. Let \([m]=\{1,2,\ldots,m\}\).

\begin{cor}\label{corkjoin}
    Let \(k \ge 2\) and \(D_1,D_2,\ldots,D_k\) be oriented graphs. Assume that there is $j\in [k]$ such that \(\mathrm{inv}(D_j)\ge 1\) and \(\mathrm{inv}(D_i)=1\)
    for all \(i \in [k]\setminus \{j\}\). Then
   $$\mathrm{inv}([D_1,D_2,\ldots,D_k])=\left\{
      \begin{array}{ll}
                 \sum_{i=1}^k \mathrm{inv}(D_i) -1 & \mbox{if~~ $\mathrm{inv}(\overrightarrow{C_3} \Rightarrow D_j)=\mathrm{inv}(D_j)$,}\\
        \sum_{i=1}^k \mathrm{inv}(D_i)  & \mbox{ otherwise.}
    \end{array}\right.$$
\end{cor}

Theorem \ref{mainresult} is actually a special case of this corollary.

\noindent \textbf{\emph{Proof of Theorem \ref{mainresult}.}}
Since $\mathrm{inv}(R)=k$ is even, we have $\mathrm{inv}(\overrightarrow{C_3}\Rightarrow R)=\mathrm{inv}(R)+1$ by Theorem \ref{thm1.3}. From Corollary \ref{corkjoin}, we are done.
   \hfill $\square$\par

Theorem \ref{direction} seems very reasonable but it is non-trivial. In fact, we believe that the direction of the dijoin operation does not affect the inversion number in general.

\begin{conjecture}\label{conj dir}
    For every oriented graphs \(L\) and \(R\), \(\mathrm{inv}(L \Rightarrow R)= \mathrm{inv}(R \Rightarrow L)\).
\end{conjecture}

Basing on Theorems~\ref{abnormal}, we also give the following conjecture.

\begin{conjecture}
    For tournaments \(T_1,T_2\) with \(\mathrm{inv}(T_1) \ge 2\),
    we have \(\mathrm{inv}(T_1 \Rightarrow T_2) > \mathrm{inv}(T_2)\).
\end{conjecture}
\vskip.2cm

There are lots of results about \(\invn\).
Belkhechine et al. \cite{BELKHECHINE2010703} first proved that
         \( \frac{n-1}{2}-\log n \le \invn \le n-3\)
    for all integer \(n \ge 4\).

Bang-Jensen et al. \cite{2105.04137} found that the constant term in the upper bound can be improved very slightly.
Alon el at. \cite{2212.11969} greatly improved the bounds of \(\invn\), showed that \(\invn = (1+o(1))n\). As the proof of it relies on probabilistic methods, there is no explicit construction for a digraph with
   large inversion number close to \(\invn\). The largest one found is \([\overrightarrow{C_3}]_{\frac{n}{3}}\) whose inversion
   number is \(\frac{n}{3}\).

Let \(Q_n\) be the tournament obtained from the transitive tournament by reversing the arcs of its unique
directed hamiltonian path \((v_1, v_2, \ldots , v_n)\).
Belkhechine el at. \cite{BELKHECHINE2010703} conjectured that
\(\mathrm{inv}(Q_n)= \lfloor \frac{n-1}{2} \rfloor\). It means that \(Q_n\)
is possibly a specific example with much larger inversion number.
It is obvious that \(\mathrm{inv}(Q_n) \le  \lfloor \frac{n-1}{2} \rfloor\) since we can give a \(\lfloor \frac{n-1}{2} \rfloor\)-decycling family \(X_i=\{v_{2i},v_{2i+1}\}\) for \(1 \le i \le
\lfloor \frac{n-1}{2} \rfloor\).

Our another main result of the paper is to give a construction with inversion number slightly more than \(\frac{n}{3}\) using blow-up graphs.

\begin{theorem}\label{thm1.5} Let
     \(T\) be a tournament of order \(n\) and \(\invt = 1\). Then \(\mathrm{inv}(T[D_1, \ldots,D_n]) =n+1\)
     for every oriented graphs \(\{D_i\}_{1 \le i \le n}\) with \(\invdi=1\). Moreover,
     \(\mathrm{inv}(T[\overrightarrow{C_3}]_k)=k+1\).
\end{theorem}

    We can prove that for oriented graphs \(D_i\) with \(\invdi=1\) ($1\le i\le n$), the  inequality
    \[
         n \le \mathrm{inv}(T[D_1,D_2,\ldots,D_n]) \le n+\mathrm{inv}(T)
    \]holds.
    Then for \(\invt\in \{0,1\}\), we have \(\mathrm{inv}(T[D_1,D_2,\ldots,D_n]) = n+\mathrm{inv}(T)\) by Theorem~\ref{thm1.5}. If \(\invt\notin \{0,1\}\), we have the following result.

    \begin{theorem}\label{example}
        Let \(k \ge 3\) be an  odd integer. Then there exists a tournament \(T_k\) with
        \(\mathrm{inv}(T_k)=k\) such that
        \[
        \mathrm{inv}(T_k[D_1,D_2,\ldots,D_{n}]) \le n+k-1
        \]
        for every oriented graphs \(\{D_i\}_{1 \le i \le n}\) with \(\invdi=1\).
        Here \(n = |V(T_k)|\) is the order of \(T_k\).
           \end{theorem}

            We have the following conjecture for even case.

    \begin{conjecture}
        Let \(k \ge 2\) be an even  integer, \(T\)  be a tournament with \(\invt=k\)  and \(|V(T)|=n\). For oriented graphs
        \(D_1,D_2,\ldots,D_n\) with \(\invdi=1\),  we have
        \[
        \mathrm{inv}(T[D_1,D_2,\cdots,D_n]) = n+k.
        \]
    \end{conjecture}

We detail some of the notation and observations to be used in this paper. If \(xy \in A(D)\), we say \(x\) dominates \(y\), denoted by \(x \rightarrow y\). If every vertex of \(A \subseteq V(D)\) dominates every vertex of \(B \subseteq V(D)\), then we say \(A\) dominates \(B\), denoted by \(A \rightarrow B\).

    Given a tournament \(T\) and the result tournament \(T'\) of \(T\) after inverting. We introduce a sign `\(<\)' to express the ordering of vertices in \(T'\),
    that is \(x < y\) if \(x y\in A(T')\). Then \(x < y<z\) means \(x y,xz,yz\in A(T')\). Let $A,B\subseteq V(T)$.
   Denote \(A < B \) (resp. \(A \le  y \)) if \(x < y\) for every \(x \in A \) and \(y \in B\) (resp.
     \(x < y\) for every \(x \in A \) with \(x \neq y\)). Note that if the result tournament \(T'\) is acyclic, then `\(<\)' is transitive and gives a total order on $V(T')$.

It is common to use vectors in $\mathbb{F}_2$ to investigate the inversion problem of graphs. For a digraph \(D\), a decycling family \((X_i)_{1 \le i \le k}\) and a vertex \(v \in V(D)\), we define the \emph{characteristic vector} \(\textbf{v} \in \mathbb{F}_2^k\) of \(v\) in \((X_i)_{1 \le i \le k}\) where the \(i\)-th element of \(\textbf{v}\) is \(1\) if and only if \(v \in X_i\). For \(\textbf{u},\textbf{v} \in \mathbb{F}_2^k\), we write
\(\textbf{u} \cdot \textbf{v}\) to be the scalar product of \(\textbf{u}\) and \(\textbf{v}\) over \(\mathbb{F}_2^k\).
We say a collection \(\{\textbf{u}_i\}_{i \in I}\) is \emph{orthonormal} if
\(\textbf{u}_i \cdot \textbf{u}_i=1\)
and \(\textbf{u}_i \cdot \textbf{u}_j=0\) for \(i \neq j\). For \(uv \in A(D)\), it is obvious that \(\textbf{u} \cdot \textbf{v} =0\) if and only if \(uv \in A(\mathrm{Inv}(D;(X_i)_{i \in I}))\).

When we study the inversion number, we mainly focus on tournaments because of the following observations.

\begin{observation}[Bang-Jensen et al. \cite{2105.04137}]\label{ob1}
    If \(D'\) is a subdigraph of \(D\), then \(\mathrm{inv}(D') \le \mathrm{inv}(D)\). Actually, if
    \((X_i)_{i \in I}\) is a decycling family of \(D\), then \((X_i \cap V(D'))_{i \in I}\)
    is a decycling family of \(D'\).
\end{observation}

\begin{observation}[Bang-Jensen et al. \cite{2105.04137}]\label{extend}If \((X_i)_{i \in I}\) is a decycling family of an oriented graph $D$, then $D$ can be extended
to a tournament $T$  such that \(\mathrm{inv}(D)=\mathrm{inv}(T)\) and \((X_i)_{i \in I}\)
    is still a decycling family of \(T\).
\end{observation}

Observation~\ref{extend} shows that we only need to consider tournaments rather than oriented graphs in lots of questions,
and it preserves the decycling family of the original graph.

After finishing the first version of this paper, we notice the independent work by Behague et al. \cite{2404.10663}. They are also working on the inversion number of dijoins and disprove a different conjecture of Aubian et al. \cite{2212.09188}. To our knowledge, their results do not imply any of our results in this paper.

The rest of this paper is organized as follows.
In Section~\ref{sec2} we give  proofs of Theorems~\ref{thm1.3} to
\ref{abnormal}.
In Section~\ref{sec3} we prove Theorem~\ref{thm1.5} and study the inversion number of the blow-up graphs.

\section{Proofs of Theorems~\ref{thm1.3} to
\ref{abnormal}}\label{sec2}

  In this section we first prove Theorem~\ref{thm1.3}. 
  \subsection{Proof of Theorem~\ref{thm1.3}}
  The proof of Theorem~\ref{thm1.3} relies on the ranks of matrices. We prove that the rank of characteristic vectors of any decycling family of an oriented graph $D$ is no less than  \(\invd\) if \(\invd\) is even. We first have the following lemma.
\begin{lemma}\label{lemma2.1}
    For any odd integer \(n\) and every symmetric matrix \(M \in M_n(\mathbb{F}_2)\), there is \(U \in M_n(\mathbb{F}_2)\) such that
    \(U^tU=M\).
\end{lemma}

\begin{proof}
    By induction on \(n\). It is trivial  when \(n=1\). Let \(M=(m_{ij})_{1 \le i,j \le n}\) be a symmetric matrix with \(n \ge 3\). We complete the proof by considering two cases.

    \textbf{Case \(1\).} There exists \(i\) such that \(m_{ii}=1\). By interchanging columns and corresponding rows
    we assume \(i=1\).

    \textbf{Case \(1.1\).} There exists \(j \neq 1\)  such that the submatrix
    \(A=\begin{pmatrix}
        m_{11} & m_{1j} \\
        m_{j1} & m_{jj}
    \end{pmatrix}\) is  non-singular. Also assume \(j=2\).

    In this subcase,
    \(M=\begin{pmatrix}
        A & B \\
        B^t & M_0
    \end{pmatrix}\). Note that
    \[
        \begin{pmatrix}
        I_2 & 0 \\
        -B^tA^{-1} & I_{n-2}
    \end{pmatrix}
    \begin{pmatrix}
        A & B \\
        B^t & M_0
    \end{pmatrix}
    \begin{pmatrix}
        I_2 & -A^{-1}B \\
        0 & I_{n-2}
    \end{pmatrix}=
    \begin{pmatrix}
        A & 0 \\
        0 & M_0-B^tA^{-1}B
    \end{pmatrix}.
    \]
    Since \(m_{11}=1\) and \(A\) is non-singular,  \(A=\begin{pmatrix}
        1 & 1 \\
        1 & 0
    \end{pmatrix}\) or \(\begin{pmatrix}
        1 & 0 \\
        0 & 1
    \end{pmatrix}\). Obviously, there is $U_1\in M_2(\mathbb{F}_2)$ such that \(U_1^tU_1=A\). By induction, there is $U_2\in M_{n-2}(\mathbb{F}_2)$ such that \(U_2^tU_2=M_0-B^tA^{-1}B\).
    Thus the result holds.

    \textbf{Case \(1.2\).} For every \(j>1\),  \(A_j=\begin{pmatrix}
        m_{11} & m_{1j} \\
        m_{j1} & m_{jj}
    \end{pmatrix}\) is singular.

    In this subcase,  \(A_j=\begin{pmatrix}
        1 & 1 \\
        1 & 1
    \end{pmatrix}\) or \(\begin{pmatrix}
        1 & 0 \\
        0 & 0
    \end{pmatrix}\) since \(M\) is symmetric.
    By elementary column transformations and corresponding row transformations we can change all
    \(\begin{pmatrix}
        1 & 1 \\
        1 & 1
    \end{pmatrix}\) to
    \(\begin{pmatrix}
        1 & 0 \\
        0 & 0
    \end{pmatrix}\). Thus we can assume
    \[
         M=\begin{pmatrix}
                1 & 0 & 0 & \cdots & 0\\
                0 & 0 & m_{23} & \dots & m_{2n} \\
                0 & m_{32} & 0 & \dots & m_{3n} \\
                \vdots & \vdots & \vdots & \ddots & \vdots \\
                0 & m_{n2} & m_{n3} & \cdots & 0\\
            \end{pmatrix}.
    \]
    Let
    \[
         C=\begin{pmatrix}
                1 & 0 & 0 & \cdots & 0\\
                0 & 1 & m_{23} & \dots & m_{2n} \\
                0 & m_{32} & 0 & \dots & m_{3n} \\
                \vdots & \vdots & \vdots & \ddots & \vdots \\
                0 & m_{n2} & m_{n3} & \cdots & 0\\
            \end{pmatrix}.
    \]
    There is \(V \in M_n(\mathbb{F}_2)\) such that \(V^tV =C\)
    by \textbf{Case} \(1.1\).
   Assume \(V=(\alpha_1,\alpha_2,\ldots,\alpha_n)\).
    Let \(U=(\mathbf{1},\alpha_1+\alpha_2,\alpha_3,\ldots,\alpha_n)\),
    where \(\mathbf{1}\) is
         the all-ones vector. Then $U\in M_n(\mathbb{F}_2)$. Note that $\mathbf{1} \cdot \mathbf{1}=1$, $\mathbf{1} \cdot \alpha_i = \alpha_i \cdot \alpha_i
    = 0 $ for $i \ge 3$, $\mathbf{1} \cdot \alpha_i = \alpha_i \cdot \alpha_i
    = 1 $ for $i \le 2$ and $\alpha_1 \cdot \alpha_i = 0$ for $ i \ge 2.$
   Hence \(U^tU =M\).

    \textbf{Case \(2\).} \(m_{ii}=0\) holds for every \(i \in [n]\).

    Let
    \[
          C_0=M+\begin{pmatrix}
                1 & 0 & 0 & \cdots & 0\\
                0 & 0 & 0 & \dots & 0 \\
                0 & 0 & 0 & \dots & 0 \\
                \vdots & \vdots & \vdots & \ddots & \vdots \\
                0 & 0 & 0 & \cdots & 0\\
            \end{pmatrix}.
    \]
    There is \(U_0 \in M_n(\mathbb{F}_2)\) such that \(U_0^tU_0= C_0\) by \textbf{Case} \(1\).
    Assume \(U_0=(\alpha_1',\alpha_2',\ldots,\alpha_n')\).
    Let \(U=(\alpha_1'+\mathbf{1},\alpha_2',\alpha_3',\ldots,\alpha_n')\). Note that $\mathbf{1} \cdot \mathbf{1}=1$  and $\mathbf{1}  \cdot \alpha_i' = \alpha_i'  \cdot \alpha_i' =0$ for all $1\le i\le n$. Hence \(M = U^t U\) and we are done.
    \end{proof}

\begin{lemma}\label{thm2.2}
    For every even integer \(k\) and a tournament \(T\) with \(\invt=k\), \(\mathrm{rank}(\Lambda) \ge k\), where \(\Lambda=\{ \mathbf{u} | u \in V(T)\}\) is the set of characteristic vectors of a decycling family of $T$.
    Moreover, \(\mathrm{rank}(\Lambda) = k\) if the decycling family of $T$ is a \(k\)-decycling family.
\end{lemma}

\begin{proof}
    Suppose for a contradiction that \(\mathrm{rank}(\Lambda) < k\).
    Then there are \(k-1\) vectors \(\{\alpha_1, \ldots , \alpha_{k-1}\}\) \( \subset \Lambda\) such that
    \(\mathbf{u}\) is a linear combination of them for each \(u \in V(T)\). Let
    \(
        \mathbf{u}=\sum_{i=1}^{k-1} \lambda_{u,i} \alpha_i
    \)
   and \(M \in M_{k-1}(\mathbb{F}_2)\)  the Gram matrix of \(\{\alpha_1, \alpha_2, \ldots , \alpha_{k-1}\}\).
    Then there exists \(U \in M_{k-1}(\mathbb{F}_2)\) such that \(M=U^tU\) by Lemma~\ref{lemma2.1}.
    Assume \(U=(\beta_1,\beta_2,\ldots,\beta_{k-1})\). Then \(\beta_i \in \mathbb{F}_2^{k-1}\) and \(\beta_i \cdot \beta_j=\alpha_i \cdot \alpha_j\) for any $i,j\in\{1,\ldots,k-1\}$. We have a
    new \((k-1)\)-decycling family given by the characteristic vector
    \[
         \mathbf{u}=\sum_{i=1}^{k-1} \lambda_{u,i} \beta_i
    \]
    for each \(u \in V(T)\), because the scalar products between vertices are fixed. A contradiction with \(\invt=k\).
\end{proof}

\begin{remark}\label{remark} The results  of  Lemma \ref{lemma2.1} only hold when $k$ is even.
    If $k$ is odd, then \(k-2\) is odd. By  Lemma~\ref{lemma2.1}, there is \(U \in M_{k-2}(\mathbb{F}_2)\) such that
    \(U^tU=M\). Hence
    we can  prove \(\mathrm{rank}(\Lambda) \ge k-1\) by using the same method.
\end{remark}

To prove Theorem~\ref{thm1.3}, we give a stronger conclusion as following.

\begin{lemma}\label{Essential claim}
    Let \(D\) be an oriented graph. If \(\mathrm{inv}(\overrightarrow{C_3} \Rightarrow D) =
    \invd = k\), then there exists a \(k\)-decycling family of \(D\) with the set of
         characteristic vectors \(\Lambda=\{\mathbf{z}| z \in V(D)\}\) such that
         \(\mathrm{rank}(\Lambda)=k-1\).
\end{lemma}

\begin{proof}
 Note that \(
    \invd = k\). By Observation~\ref{extend}, Lemma~\ref{thm2.2}
    and Remark~\ref{remark}, we can
 assume \(\mathrm{rank}(\Lambda)=k\) for every
    \(k\)-decycling family of $D$.

     Assume \(V(\overrightarrow{C_3})=\{u,v,w\}\). Fix a \(k\)-decycling family \((X_i)_{1 \le i \le k}\) of \(\overrightarrow{C_3} \Rightarrow D\).
    We extend \(D\) to a tournament \(T\) such that \(\overrightarrow{C_3} \Rightarrow T\) has the same \(k\)-decycling family and \(\invt = \invd =k\) by Observation~\ref{extend}.
     After the inversions, assume the new ordering of \(V(\overrightarrow{C_3} \Rightarrow T)\) is
    \[
        P < u < Q < v < R < w < S,
    \]
    where \(\{P,Q,R,S\}\) is a partition of \(V(T)\) with possibly empty set. Comparing with the old ordering of \(V(\overrightarrow{C_3} \Rightarrow T)\), that is
    \[
        \{u,v,w\} \rightarrow P \cup Q \cup R \cup S,
    \]
    we can give the scalar products between \(\mathbf{u},\mathbf{v},\mathbf{w} \in \fk\) and \(\mathbf{p},\mathbf{q},\mathbf{r},\mathbf{s} \in \fk\) as following, where \(p \in P, q \in Q, r\in R, s \in S\).
    \begin{equation*}\tag{1}	
        \begin{aligned}
	      &\mathbf{p} \cdot \mathbf{u}=1, \quad \mathbf{q} \cdot \mathbf{u}=0, \quad \mathbf{r} \cdot \mathbf{u}=0, \quad \mathbf{s} \cdot \mathbf{u}=0, \\
       &\mathbf{p} \cdot \mathbf{v}=1, \quad \mathbf{q} \cdot \mathbf{v}=1, \quad \mathbf{r} \cdot \mathbf{v}=0, \quad \mathbf{s} \cdot \mathbf{v}=0, \\
       &\mathbf{p} \cdot \mathbf{w}=1, \quad \mathbf{q} \cdot \mathbf{w}=1, \quad \mathbf{r} \cdot \mathbf{w}=1, \quad \mathbf{s} \cdot \mathbf{w}=0.
        \end{aligned}
    \end{equation*}
   Before the inversions, three arcs of \(\overrightarrow{C_3}\) are either
    \(u \rightarrow v\), \(v \rightarrow w \), \(w\rightarrow u\) or \(u \rightarrow w\), \(w \rightarrow v \), \(v\rightarrow u\). Then we have either $\mathbf{u} \cdot \mathbf{v}=0$, $\mathbf{v} \cdot \mathbf{w}=0$ and $\mathbf{u} \cdot \mathbf{w}=1$, or $\mathbf{u} \cdot \mathbf{v}=1$, $\mathbf{v} \cdot \mathbf{w}=1$ and $\mathbf{u} \cdot \mathbf{w}=0$.
          In both cases, we have \(\mathbf{u} \cdot (\mathbf{v}+\mathbf{w})=1\) and \(\mathbf{w} \cdot (\mathbf{u}+\mathbf{v})=1\). Thus \(\mathbf{u}, \mathbf{v}+\mathbf{w}, \mathbf{u}+\mathbf{v}\)
     are nonzero vectors.

      Since \(    \invt = k\), \(\Lambda= \{\mathbf{p}| p \in P\} \cup \{\mathbf{q}| q \in Q\} \cup\{\mathbf{r}| r \in R\} \cup\{\mathbf{s}| s \in S\}\) is the set of characteristic vectors of the \(k\)-decycling family \((X_i\cap V(T))_{1 \le i \le k}\), of \(T\). By assumption, \(\mathrm{rank}(\Lambda)=k\).
     If \(P= \emptyset\) (resp. \(Q= \emptyset\) or \(R= \emptyset\)), then \(\mathbf{u} \bot \Lambda\) (resp. \((\mathbf{u}+\mathbf{v}) \bot \Lambda\) or
      \((\mathbf{v}+\mathbf{w}) \bot \Lambda\)), a contradiction with  \(\mathrm{rank}(\Lambda)=k\). Hence we have
     $P $, $Q $ and $R $ all are non-empty sets. Note that the digraph Inv\((T;(X_i\cap V(T))_{i \in I})\) is acyclic and the ordering of $V(T)$ in Inv\((T;(X_i\cap V(T))_{i \in I})\) is $P  < Q < R  < S.$ Particularly, Inv\((T\left\langle Y \right\rangle;(X_i\cap Y)_{i \in I})\) is acyclic for any $Y\in\{P,Q,R,S\}$.

     Now we  choose \(p_0 \in P, q_0 \in Q, r_0 \in R\). From \((1)\), it is not difficult to show that  \(\{\mathbf{p_0}, \mathbf{q_0}, \mathbf{r_0}\}\)  and \(\{\mathbf{u}, \mathbf{v}, \mathbf{w}\}\) are linear independent.
      Let \(\Phi=\{ \mathbf{x} \in \fk | \mathbf{x} \bot \mathbf{u}, \mathbf{x} \bot \mathbf{v}, \mathbf{x} \bot \mathbf{w}\}\). Then \(\mathrm{rank}(\Phi)=k-3\). For any \(p \in P\), we have \(\mathbf{p}+\mathbf{p_0} \in \Phi\) which implies
       \(\{\mathbf{p}|p \in P\} \subset \mathbf{p_0}+\Phi\). So for any \(p\in P\), we can set
      \(
             \mathbf{p}=\mathbf{p_0}+\phi_p,
      \)
      where \(\phi_p \in \Phi\). Similarly, for any \(q \in Q, r \in R\), we set
      \(
             \mathbf{q}=\mathbf{q_0}+\phi_q, ~ \mathbf{r}=\mathbf{r_0}+\phi_r,
      \)
      where \(\phi_q,\phi_r \in \Phi\). 
           Note that any vector is a linear combination of
      \(\{\mathbf{p_0}, \mathbf{q_0}, \mathbf{r_0}\}\) and a vector in \(\Phi\). Assume
      \[
           \mathbf{u}+\mathbf{v}=\lambda_1\mathbf{p_0}+\lambda_2\mathbf{q_0}+\lambda_3\mathbf{r_0}+\phi.
      \]
      Since \(\mathbf{w} \cdot (\mathbf{u}+\mathbf{v})=1, \mathbf{u}+\mathbf{v} \notin \Phi\).
      Then  \(|\{\lambda_1,\lambda_2,\lambda_3\}\cap(\mathbb{F}_2\setminus\{0\})|\ge 1\). We
      consider the following three cases.

      \textbf{Case \(1\).}  \(\lambda_1=1\).

      In this case, \(\tau:= \mathbf{p_0}+\mathbf{u}+\mathbf{v}\) is a
      linear combination of \(\mathbf{q_0},\mathbf{r_0}\) and a vector in \(\Phi\). Let
      $$\Lambda'=\{\mathbf{p}+\mathbf{u}+\mathbf{v}| p \in P\} \cup \{\mathbf{q}| q \in Q\} \cup\{\mathbf{r}| r \in R\} \cup\{\mathbf{s}| s \in S\}.$$Then \(\Lambda'\) is the set of characteristic vectors of a $k$-family, say  \((X'_i)_{1 \le i \le k}\), of $T$. Obviously, \(\mathrm{rank}(\Lambda')=k-1\). So we just need to show that  \((X'_i)_{1 \le i \le k}\) is a \(k\)-decycling family of $T$, that is Inv\((T;(X'_i)_{i \in I})\) is acyclic. By comparing $\Lambda$ and $\Lambda'$, we have Inv\((T\left\langle Y \right\rangle;(X'_i\cap Y)_{i \in I})\) is acyclic for any $Y\in\{Q,R,S\}$ and $Q<R<S$ in Inv\((T;(X'_i)_{i \in I})\).

      Since \((\mathbf{u}+\mathbf{v}) \cdot \mathbf{r}=(\mathbf{u}+\mathbf{v}) \cdot \mathbf{s}=0\), we have $(\mathbf{p}+ \mathbf{u}+\mathbf{v})\cdot \mathbf{x}=\mathbf{p}\cdot \mathbf{x}$ for all $p\in P$ and $x\in R\cup S$ which implies $P<R<S$ in Inv\((T;(X'_i)_{i \in I})\). Since \((\mathbf{u}+\mathbf{v}) \cdot \mathbf{q}=1\), we have \((\mathbf{p}+ \mathbf{u}+\mathbf{v})\cdot \mathbf{q}=\mathbf{p}\cdot \mathbf{q}+1\) for  all $p\in P$ and $q\in Q$ which implies $Q<P$ in Inv\((T;(X'_i)_{i \in I})\). For any \(p_1, p_2 \in P\), we have
      \[
          (\mathbf{p_1}+ \mathbf{u}+\mathbf{v}) \cdot (\mathbf{p_2}+ \mathbf{u}+\mathbf{v})=(\mathbf{p_1} \cdot \mathbf{p_2})
          +(\mathbf{u}+\mathbf{v}) \cdot (\mathbf{u}+\mathbf{v}).
      \]
      Note that \(\mathbf{u}+\mathbf{v}\) is fixed, so the arcs in Inv\((T\left\langle P \right\rangle;(X_i\cap P)_{i \in I})\) are either all inverted or not which implies Inv\((T\left\langle P \right\rangle;(X'_i\cap P)_{i \in I})\) is acyclic.
       As a conclusion, Inv\((T\left\langle Y \right\rangle;(X'_i\cap Y)_{i \in I})\) is acyclic for all $Y\in\{P,Q,R,S\}$ and $Q < P < R < S$ in Inv\((T;(X'_i)_{i \in I})\).
       Thus \((X'_i)_{1 \le i \le k}\) is a \(k\)-decycling family of $T$.

      \textbf{Case \(2\).}  \(\lambda_2=1\).

      In this case,  \(\tau:= \mathbf{q_0}+\mathbf{u}+\mathbf{v}\) is a
      linear combination of \(\mathbf{p_0},\mathbf{r_0}\) and a vector in \(\Phi\). Let
      $$\Lambda'=\{\mathbf{p}| p \in P\} \cup \{\mathbf{q}+\mathbf{u}+\mathbf{v}| q \in Q\} \cup\{\mathbf{r}| r \in R\} \cup\{\mathbf{s}| s \in S\}.$$Then \(\Lambda'\) is the set of characteristic vectors of a $k$-family, say  \((X'_i)_{1 \le i \le k}\), of $T$. Obviously, \(\mathrm{rank}(\Lambda')=k-1\). Since \((\mathbf{u}+\mathbf{v}) \cdot \mathbf{p}=(\mathbf{u}+\mathbf{v}) \cdot \mathbf{r}=(\mathbf{u}+\mathbf{v}) \cdot \mathbf{s}=0\), by the same argument as that of Case 1, we have Inv\((T\left\langle Y \right\rangle;(X'_i\cap Y)_{i \in I})\) is acyclic for all $Y\in\{P,Q,R,S\}$ and $P < Q < R < S$ in Inv\((T;(X'_i)_{i \in I})\).
       Thus \((X'_i)_{1 \le i \le k}\) is a \(k\)-decycling family of $T$.

      \textbf{Case \(3\).}  \(\lambda_3=1\).

      In this case,   \(\tau:= \mathbf{r_0}+\mathbf{u}+\mathbf{v}\) is a
      linear combination of \(\mathbf{p_0},\mathbf{q_0}\) and a vector in \(\Phi\). Let
      $$\Lambda'=\{\mathbf{p}| p \in P\} \cup \{\mathbf{q}| q \in Q\} \cup\{\mathbf{r}+\mathbf{u}+\mathbf{v}| r \in R\} \cup\{\mathbf{s}| s \in S\}.$$Then \(\Lambda'\) is the set of characteristic vectors of a $k$-family, say  \((X'_i)_{1 \le i \le k}\), of $T$. Obviously, \(\mathrm{rank}(\Lambda')=k-1\). Note that \((\mathbf{u}+\mathbf{v}) \cdot \mathbf{p}=(\mathbf{u}+\mathbf{v}) \cdot \mathbf{s}=0\) and \((\mathbf{u}+\mathbf{v}) \cdot \mathbf{q}=1\).
      By the same argument as that of Case 1, we have Inv\((T\left\langle Y \right\rangle;(X'_i\cap Y)_{i \in I})\) is acyclic for all $Y\in\{P,Q,R,S\}$ and $P < R< Q < S$ in Inv\((T;(X'_i)_{i \in I})\).
       Thus \((X'_i)_{1 \le i \le k}\) is a \(k\)-decycling family of $T$.
  \end{proof}

  \vskip.2cm
  Now we are going to prove Theorem~\ref{thm1.3}.

 \noindent \textbf{\emph{Proof of Theorem~\ref{thm1.3}.}}
     By Observation~\ref{ob1}, \(\mathrm{inv}(\overrightarrow{C_3} \Rightarrow D) \ge k\). As we can
     invert \(\overrightarrow{C_3}\) and \(D\) respectively to transform $\overrightarrow{C_3} \Rightarrow D$ into an acyclic digraph, we have
     \(\mathrm{inv}(\overrightarrow{C_3} \Rightarrow D) \le k+1\).
     If \(\mathrm{inv}(\overrightarrow{C_3} \Rightarrow D) = k\), then there exists a \(k\)-decycling family of \(D\) with the set of
         characteristic vectors \(\Lambda=\{\mathbf{u}| u \in V(D)\}\) such that
         \(\mathrm{rank}(\Lambda)=k-1\) by Lemma~\ref{Essential claim}, which contradicts with
         Lemma~\ref{thm2.2}. Hence we have \(\mathrm{inv}(\overrightarrow{C_3} \Rightarrow D) = k+1\)
         as required.
\hfill $\square$\par

\subsection{Proofs of Theorems~\ref{direction} to
\ref{abnormal}}
We observe that in the proof of Theorem~\ref{thm1.3}, we only use the property that the rank of
the characteristic vectors and the inversion number are the same.
Hence we consider applying the conclusion to general cases.

\begin{lemma}\label{even matrix}
    For every even integer \(n\) and a symmetric matrix \(M \in M_n(\mathbb{F}_2)\), there is \(U \in M_n(\mathbb{F}_2)\) such that
    \(M =U^tU\)  if and only if
    \(M\) has a non-zero diagonal entry or \(M\) is a singular matrix.
\end{lemma}

\begin{proof}
    We prove by induction on \(n\). It is trivial when \(n=2\). Assume \(M=(m_{ij})_{1 \le i,j \le n}\) with \(n \ge 4\). We consider three cases.

    \textbf{Case \(1\).} There exists \(i\in [n]\) such that \(m_{ii}=1\). By interchanging columns and corresponding rows
    we assume \(i=1\). Let
    \(M=\begin{pmatrix}
        1 & B \\
        B^t & M_0
    \end{pmatrix}\). Note that
    \[
        \begin{pmatrix}
        1 & 0 \\
        -B^tA^{-1} & I_{n-1}
    \end{pmatrix}
    \begin{pmatrix}
        1 & B \\
        B^t & M_0
    \end{pmatrix}
    \begin{pmatrix}
        1 & -A^{-1}B \\
        0 & I_{n-1}
    \end{pmatrix}=
    \begin{pmatrix}
        1 & 0 \\
        0 & M_0-B^tB
    \end{pmatrix}.
    \]
   Since \(M_0-B^tB \in M_{n-1}(\mathbb{F}_2) \), by Lemma~\ref{lemma2.1}, \(M_0-B^tB \) has a decomposition which implies the result holds.

   \textbf{Case \(2\).} \(m_{ii}=0\) for every \(i \in [n]\) and \(M\) is a non-singular symmetric matrix.
   Suppose there is  \(U \in M_n(\mathbb{F}_2)\) such that \(U^tU=M\). Assume \(U=(\alpha_1,\alpha_2,\ldots,\alpha_n)\). Since \(\mathbf{1} \cdot \alpha_j
    = \alpha_j \cdot \alpha_j = m_{jj}=0\), for any $j\in [n]$, the vector $\alpha_j$ has an even number of 1s. So \(\mathrm{rank}(U)<n\) which implies
     \(\mathrm{rank}(M) \le \mathrm{rank}(U) <n\),
   a contradiction.

    \textbf{Case \(3\).} \(m_{ii}=0\) for every \(i \in [n]\), and \(M\) is a singular symmetric matrix.
    Then we can change the first column and first row to all zeros by elementary column transformations
    and corresponding row transformations.
    Thus we can assume \(M=\begin{pmatrix}
        0 & 0 \\
        0 & M'
    \end{pmatrix}\). Since \(M' \in M_{n-1}(\mathbb{F}_2) \), by Lemma~\ref{lemma2.1}, \(M'\) has a decomposition which implies the result holds.
\end{proof}

\begin{lemma}\label{odd invertion}
    Let \(D\) be an oriented graph. If \(\mathrm{inv}(D)=k \ge 3\), then the following propositions are equivalent.
     \begin{enumerate}
         \item[\((1)\)] \(\mathrm{inv}(\overrightarrow{C_3} \Rightarrow D) = k\).
         \item[\((2)\)] \(k\) is odd and there exists a \(k\)-decycling family of \(D\) with the set of
         characteristic vectors \(\Lambda=\{\mathbf{z}| z \in V(D)\}\) such that
         \(\Lambda \subset \{x \in \fk | x \bot \mathbf{1}\}\).
     \end{enumerate}
\end{lemma}
\begin{proof}
    By Observation~\ref{extend}, we assume \(D\) is a tournament.

    \((1) \Rightarrow (2)\). Obviously \(k\) is odd by Theorem~\ref{thm1.3}.
    By Lemma~\ref{Essential claim}, there exists a \(k\)-decycling family of \(D\) with the set of
         characteristic vectors \(\Lambda=\{\mathbf{z}| z \in V(D)\}\) such that
         \(\mathrm{rank}(\Lambda)=k-1\).

 Under the fixed decycling family, we claim \(\mathbf{z} \cdot \mathbf{z}=0\)
    for each \(z \in V(D)\), then the conclusion can be deduced  since
    \(\mathbf{1} \cdot \mathbf{z}=\mathbf{z} \cdot \mathbf{z}=0\).
    Suppose there is \(z_1 \in V(D)\) such that \(\mathbf{z_1} \cdot \mathbf{z_1}=1\). Then we can choose
    \(z_2,z_3,\ldots,z_{k-1}\) such that \(\mathbf{z_1},\mathbf{z_2},\ldots,\mathbf{z_{k-1}}\) forms a base of \(\Lambda\). So  for each \(z \in V (D )\), we have
    \(
        \mathbf{z}=\sum_{i=1}^{k-1} \lambda_{z,i} \mathbf{z_i},
    \) where $\lambda_{z,i}\in \mathbb{F}_2$.

    Let \(M \in M_{k-1}(\mathbb{F}_2)\) be the Gram matrix of \(\{\mathbf{z_1}, \mathbf{z_2}, \ldots , \mathbf{z_{k-1}}\}\).
    Since \(\mathbf{z_1} \cdot \mathbf{z_1}=1\),  there exists \(U \in M_{k-1}(\mathbb{F}_2)\) such that \(M=U^tU\) by Lemma~\ref{even matrix}.
    Let \(U=
        (\mathbf{z_1^{\prime}} , \mathbf{z_2^{\prime}} , \ldots , \mathbf{z_{k-1}^{\prime}})
    \). Then \(\mathbf{z_i^{\prime}} \cdot \mathbf{z_j^{\prime}}=\mathbf{z_i} \cdot \mathbf{z_j}\). Now we have a
    new \((k-1)\)-decycling family given by the characteristic vector
    \[
         \mathbf{z}=\sum_{i=1}^{k-1} \lambda_{z,i} \mathbf{z_i^{\prime}}
    \]
    for each \(z \in V(D)\), because the scalar products between vertices are fixed. Note that \(\mathbf{z_i^{\prime}} \in \mathbb{F}_2^{k-1}\) and this contradicts with \(\mathrm{inv}(D)=k\).

    \((2) \Rightarrow (1)\). Obviously, \(\mathrm{inv}(\overrightarrow{C_3} \Rightarrow D) \ge \mathrm{inv}(D)= k\). Let \(V(\overrightarrow{C_3})=\{u,v,w\}\) and define \(\mathbf{u}=0, \mathbf{v}
    =\mathbf{w}=\mathbf{1} \in \fk\).
    It is easy to check that \(\{\mathbf{u},\mathbf{v},\mathbf{w}\} \cup \Lambda\) is the
    set of characteristic vectors of a \(k\)-decycling family of \(\overrightarrow{C_3} \Rightarrow D\) which implies \(\mathrm{inv}(\overrightarrow{C_3} \Rightarrow D) \le  k\).
\end{proof}

\vskip.2cm

\noindent \textbf{\emph{Proof of Theorem~\ref{direction}.}}
  Let \(D^-\) be the digraph obtained by reversing all arcs of \(D\). Then \(D\) and \(D^-\) share the same decycling family and
    \(\invd= \mathrm{inv}(D^-) \).

    Assume \(\mathrm{inv}(\overrightarrow{C_3} \Rightarrow D) = \invd=k\). By Lemma~\ref{odd invertion} and the fact that \(D\) and \(D^-\) share the same decycling family,
we have  \(k\) is odd and there exists a \(k\)-decycling family of \(D^-\) with the set of
  characteristic vectors
  \(\Lambda=\{\mathbf{z}| z \in V(D^-)\}\) such that
          \(\Lambda \subset \{x \in \fk | x \bot \mathbf{1}\}\). By Lemma~\ref{odd invertion} again, \(\mathrm{inv}(\overrightarrow{C_3} \Rightarrow D^-) = \mathrm{inv}(D^-)\). Since \(\mathrm{inv}(\overrightarrow{C_3} \Rightarrow D^-) =\mathrm{inv}(D \Rightarrow \overrightarrow{C_3})\) and \(\invd= \mathrm{inv}(D^-) \), we have $\mathrm{inv}(D \Rightarrow \overrightarrow{C_3})=\mathrm{inv}(\overrightarrow{C_3} \Rightarrow D)$.

          Conversely, assume \(\mathrm{inv}(D \Rightarrow \overrightarrow{C_3})=\mathrm{inv}(\overrightarrow{C_3} \Rightarrow D^-) =\mathrm{inv}(D^-)\). By the same argument as above, we have $\mathrm{inv}(\overrightarrow{C_3} \Rightarrow D)=\mathrm{inv}(D \Rightarrow \overrightarrow{C_3})$.

 Since \(\mathrm{inv}(\overrightarrow{C_3} \Rightarrow D), \mathrm{inv}(D \Rightarrow \overrightarrow{C_3})
  \in \{\invd,\invd+1\}\), we have \(\mathrm{inv}(\overrightarrow{C_3} \Rightarrow D)=
  \mathrm{inv}(D \Rightarrow \overrightarrow{C_3})\).
\hfill $\square$\par

\begin{lemma}\label{cor ext}
    Let \(D\) be an oriented graph with \(\mathrm{inv}(D)=1\). Then
    \(\mathrm{inv}(\overrightarrow{C_3} \Rightarrow H)=\mathrm{inv}(D \Rightarrow H)\) and \(\mathrm{inv}(H \Rightarrow \overrightarrow{C_3})=\mathrm{inv}(H \Rightarrow D)\) for
    every oriented graph \(H\).
\end{lemma}

\begin{proof}
 We just show that \(\mathrm{inv}(\overrightarrow{C_3} \Rightarrow H)=\mathrm{inv}(D \Rightarrow H)\) holds for every oriented graph \(H\). Actually, we have the conclusion in the other direction if we reverse all arcs.
    Obviously \(\mathrm{inv}(\overrightarrow{C_3} \Rightarrow H), \mathrm{inv}(D \Rightarrow H)
  \in \{\mathrm{inv}(H),\mathrm{inv}(H)+1\}\).

    If \(\mathrm{inv}(\overrightarrow{C_3} \Rightarrow H)=\mathrm{inv}(H)+1\), the result holds since \(\overrightarrow{C_3} \Rightarrow H\)
    is a subdigraph of \(D \Rightarrow H\).

    If \(\mathrm{inv}(\overrightarrow{C_3} \Rightarrow H)=\mathrm{inv}(H)=k\), then \(k\) is odd by Theorem~\ref{thm1.3}.
    By Lemma~\ref{Essential claim}, there exists a \(k\)-decycling family of \(H\) with the set of
         characteristic vectors \(\Lambda=\{\mathbf{z}| z \in V(H)\}\) such that
         \(\mathrm{rank}(\Lambda)=k-1\).
    Then we obtain a \((k+1)\)-decycling family of \(D \Rightarrow H\) with the set of
         characteristic vectors \(\Lambda'=\{\mathbf{z}| z \in V(D \Rightarrow H)\}\) such that
         \(\mathrm{rank}(\Lambda')\le k\) if we invert \(D\) and \(H\) respectively.
    Obviously \(\mathrm{inv}(D \Rightarrow H)=k\) or \(k+1\).
    If \(\mathrm{inv}(D \Rightarrow H)=k+1\), \(\mathrm{rank}(\Lambda')\ge k+1\) by Lemma~\ref{thm2.2},  a contradiction.

    Then
    \(\mathrm{inv}(D \Rightarrow H)=k=\mathrm{inv}(\overrightarrow{C_3} \Rightarrow H)\) which proves the result.
\end{proof}


Now we are going to  prove Theorem~\ref{abnormal}.

\noindent \textbf{\emph{Proof of Theorem~\ref{abnormal}.}}
By contradiction. Assume \(D\) is a tournament by Observation~\ref{extend}.
    Suppose \(\mathrm{inv}([\overrightarrow{C_3} , \overrightarrow{C_3} , D]) =  \mathrm{inv}(\overrightarrow{C_3} \Rightarrow D)=k\). Then \(k\) is odd  by Theorem~\ref{thm1.3}.
    By Lemma~\ref{odd invertion}, there exists a \(k\)-decycling family \((X_i)_{1 \le i \le k}\) of \(\overrightarrow{C_3} \Rightarrow D\) with the set of
         characteristic vectors \(\Lambda=\{\mathbf{z}| z \in V(\overrightarrow{C_3} \Rightarrow D)\}\) such that
         \(\Lambda \subset F\) where  \(F=\{x \in \fk | x \bot \mathbf{1}\}\). Then \(\mathrm{rank}(F)=k-1\).  Let \(\Gamma=\{\mathbf{x}\in \Lambda| x \in V(D)\}\) be the set of the characteristics vectors of the $k$-decycling family \((X_i\cap V(D))_{1 \le i \le k}\) of $D$.
         Since $\mathrm{inv}(D)$ is either $k$ or $k-1$, \(\mathrm{rank}(\Gamma)\ge k-1\) by Remark \ref{remark} and Lemma \ref{thm2.2}.
                  By Remark \ref{remark} and \(\Lambda \subset F\), \(\mathrm{rank}(\Lambda)=k-1\). From \(\Gamma \subset \Lambda\), we have \(\mathrm{rank}(\Gamma)=k-1\).  Assume \(V(\overrightarrow{C_3})=\{u,v,w\}\).
    Under the fixed decycling family, assume the
       new ordering of \(V(\overrightarrow{C_3} \Rightarrow D)\) is
    \[
        P < u < Q < v < R < w < S,
    \]
    where \(P,Q,R,S\) form a partition of \(V(D)\) with possibly empty set. Then $\Gamma=\{\mathbf{p}| p \in P\} \cup \{\mathbf{q}| q \in Q\} \cup\{\mathbf{r}| r \in R\} \cup\{\mathbf{s}| s \in S\}$. Comparing with the original tournament
    \[
        \{u,v,w\} \rightarrow P \cup Q \cup R \cup S,
    \]
    we can give the scalar products between \(\mathbf{u},\mathbf{v},\mathbf{w} \in \fk\) and \(\mathbf{p},\mathbf{q},\mathbf{r},\mathbf{s} \in \fk\), where \(p \in P, q \in Q, r\in R, s \in S\). That is
    \begin{equation*}\tag{2}	
        \begin{aligned}
	      &\mathbf{p} \cdot \mathbf{u}=1, \quad \mathbf{q} \cdot \mathbf{u}=0, \quad \mathbf{r} \cdot \mathbf{u}=0, \quad \mathbf{s} \cdot \mathbf{u}=0, \\
       &\mathbf{p} \cdot \mathbf{v}=1, \quad \mathbf{q} \cdot \mathbf{v}=1, \quad \mathbf{r} \cdot \mathbf{v}=0, \quad \mathbf{s} \cdot \mathbf{v}=0, \\
       &\mathbf{p} \cdot \mathbf{w}=1, \quad \mathbf{q} \cdot \mathbf{w}=1, \quad \mathbf{r} \cdot \mathbf{w}=1, \quad \mathbf{s} \cdot \mathbf{w}=0.
        \end{aligned}
    \end{equation*}
   By the same argument as the proof in Lemma \ref{Essential claim}, we have \(\mathbf{u} \cdot (\mathbf{v}+\mathbf{w})=1\) and \(\mathbf{w} \cdot (\mathbf{u}+\mathbf{v})=1\) which implies that \(\mathbf{u}, \mathbf{v}+\mathbf{w}, \mathbf{u}+\mathbf{v}\)
     are nonzero vectors.

     If \(P= \emptyset\) (resp. \(Q= \emptyset\) or \(R= \emptyset\)), then \(\mathbf{u} \bot \Gamma\) (resp. \((\mathbf{u}+\mathbf{v}) \bot \Gamma\) or
      \((\mathbf{v}+\mathbf{w}) \bot \Gamma\)). Note that $\mathbf{u}$ (resp. $\mathbf{u}+\mathbf{v}$  or $\mathbf{v}+\mathbf{w}$) is linear independent with $\mathbf{1}$ and $\Gamma\subset  F $. Thus we have \(\mathrm{rank}(\Gamma) \le k-2\), a contradiction with  \(\mathrm{rank}(\Gamma)=k-1\). Hence we have
     $P $, $Q $ and $R $ all are non-empty sets.

     Now we can choose \(p_0 \in P, q_0 \in Q, r_0 \in R\). From \((2)\), it is not difficult to show that  \(\{\mathbf{p_0}, \mathbf{q_0}, \mathbf{r_0}\}\)  and \(\{\mathbf{u}, \mathbf{v}, \mathbf{w}\}\) are linear independent.
      Let \(\Phi=\{ \mathbf{x} \in F | \mathbf{x} \bot \mathbf{u}, \mathbf{x} \bot \mathbf{v}, \mathbf{x} \bot \mathbf{w}\}\). Since \(\mathrm{rank}(\{\mathbf{1},\mathbf{u}, \mathbf{v}, \mathbf{w}\})=4\),  \(\mathrm{rank}(\Phi)=k-4\). By the same argument as  the proof in Lemma \ref{Essential claim},  for any \(p \in P\) (resp. $q \in Q$ or $r \in R$), there is $\phi_p\in \Phi$ ( resp. $\phi_q\in \Phi$ or $\phi_r\in \Phi$) such that $\mathbf{p}=\mathbf{p_0}+\phi_p$ (resp. $\mathbf{q}=\mathbf{q_0}+\phi_q$ or $\mathbf{r}=\mathbf{r_0}+\phi_r$).
       Note that any vector is a linear combination of
      \(\{\mathbf{p_0}, \mathbf{q_0}, \mathbf{r_0}\}\) and a vector in \(\Phi\). Assume
      \[
           \mathbf{u}+\mathbf{v}=\lambda_1\mathbf{p_0}+\lambda_2\mathbf{q_0}+\lambda_3\mathbf{r_0}+\phi.
      \]
      Since \(\mathbf{w} \cdot (\mathbf{u}+\mathbf{v})=1, \mathbf{u}+\mathbf{v} \notin \Phi\).
      Then  \(|\{\lambda_1,\lambda_2,\lambda_3\}\cap(\mathbb{F}_2\setminus\{0\})|\ge 1\). Let
            $$\Gamma'=\left\{
      \begin{array}{ll}
      \{\mathbf{p}+\mathbf{u}+\mathbf{v}| p \in P\} \cup \{\mathbf{q}| q \in Q\} \cup\{\mathbf{r}| r \in R\} \cup\{\mathbf{s}| s \in S\}& \mbox{if $\lambda_1=1$,}\\
      \{\mathbf{p}| p \in P\} \cup \{\mathbf{q}+\mathbf{u}+\mathbf{v}| q \in Q\} \cup\{\mathbf{r}| r \in R\} \cup\{\mathbf{s}| s \in S\}& \mbox{if $\lambda_2=1$,}\\
      \{\mathbf{p}| p \in P\} \cup \{\mathbf{q}| q \in Q\} \cup\{\mathbf{r}+\mathbf{u}+\mathbf{v}| r \in R\} \cup\{\mathbf{s}| s \in S\}& \mbox{if $\lambda_3=1$.}
      \end{array}\right.$$

      By the same argument as  the proof in Lemma \ref{Essential claim}, we have \(\Gamma'\) is the set of characteristic vectors of a \(k\)-decycling family of $D$. Obviously, \(\mathrm{rank}(\Gamma')=k-2\), we have a contradiction because the rank of the characteristics vectors of any decycling family of $D$ is at least $k-1$.\hfill $\square$\par

\noindent \textbf{\emph{Proof of Corollary~\ref{corkjoin}.}}
We prove it by induction on $k$. The situation when \(k=2\) is proved by Lemma~\ref{cor ext}. Assume $k\ge 3$. By Lemma~\ref{cor ext} and Theorem \ref{direction}, we can assume $j\not=1$ and then
 $\mathrm{inv}([D_1,D_2,\ldots,D_k])=\mathrm{inv}([\overrightarrow{C_3},D_2,\ldots,D_k])$ by Lemma~\ref{cor ext}. If $j\not=k$,  $\mathrm{inv}([\overrightarrow{C_3},D_2,\ldots,D_k])=\mathrm{inv}([\overrightarrow{C_3},D_2,\ldots,D_{k-1},\overrightarrow{C_3}])$ by Lemma~\ref{cor ext}. If $j=k$, $\mathrm{inv}([\overrightarrow{C_3},D_2,\ldots,D_k])=\mathrm{inv}([D_2,\ldots,D_{k},\overrightarrow{C_3}])$ by Theorem~\ref{direction}. By Lemma~\ref{cor ext}, $\mathrm{inv}([D_2,\ldots,D_{k},\overrightarrow{C_3}])=\mathrm{inv}([\overrightarrow{C_3},D_3,\ldots,D_k,\overrightarrow{C_3}])$.

 Hence we can assume $\mathrm{inv}([D_1,D_2,\ldots,D_k])=\mathrm{inv}([\overrightarrow{C_3},D_2,\ldots,D_{k-1},\overrightarrow{C_3}])$.
By Theorems~\ref{direction}
 and \ref{abnormal}, for any oriented graph \(D\),
     \(\mathrm{inv}([\overrightarrow{C_3},  \overrightarrow{C_3},  D]) =
     \mathrm{inv}([\overrightarrow{C_3},  D, \overrightarrow{C_3}] )=
     \mathrm{inv}([D, \overrightarrow{C_3}, \overrightarrow{C_3}])=
     \mathrm{inv}(\overrightarrow{C_3} \Rightarrow D)+1\). Then we obtain the result as required.
\hfill $\square$\par

\section{Blow-up graphs}\label{sec3}
    The content in this section is inspired by Alon el at. \cite{2212.11969}. They mainly concerned about
    the graph \(D=TT_k[D_1,D_2,\ldots,D_k]\).
    We generalize the base graph \(TT_k\) to any tournament \(H\),
    since it creates a larger inversion number sometimes.

\noindent \textbf{\emph{Proof of Theorem~\ref{thm1.5}.}}
By contradiction. Assume  \(D_i\) is tournament for all $1\le i \le n$.
        First \(\mathrm{inv}(T[D_1,D_2, \cdots,D_n]) \le n+1\) as we can invert \(T\) to acyclic and
        then invert \(D_i\) one after another.
        Since \(T[\overrightarrow{C_3}]_n\) is a subdigraph of \(T[D_1,D_2, \ldots,D_n]\), we only need to prove  the
        situation when \(D_i=\overrightarrow{C_3}\) by Observation \ref{ob1}. Let \(\mathrm{inv}(T[D_1,D_2, \ldots,D_n]) =k\). For short, set $G=T[D_1,D_2, \ldots,D_n]$. Then there exists a \(k\)-decycling family $(X_i)_{i \in I}$ of \(G\) with the set of
         characteristic vectors \(\Lambda=\{\mathbf{u}\in \fk| u \in V(G)\}\).  Note that the digraph Inv\((G;(X_i)_{i \in I})\) is acyclic.

         Recall $u<v$ (resp. $u\rightarrow v$) means $uv\in A(Inv(G;(X_i)_{i \in I}) $ (resp. $uv\in A(G)$).

         Let \(V(D_i)=\{u_i,v_i,w_i\}\), $1\le i\le n$. Since Inv\((D_i;(X_j)_{j \in I})\) is acyclic, we can assume \(u_i < v_i, u_i < w_i\) and \(\mathbf{u_i} \cdot \mathbf{v_i}=0, \mathbf{u_i} \cdot \mathbf{w_i}=1\),
         which means \(v_i \rightarrow w_i\) for $1\le i\le n$.

         \begin{claim}\label{lemma3.1}
             \(\{\mathbf{u_i}\}_{1 \le i \le n}\) are linearly independent.
         \end{claim}

         \noindent \textbf{\emph{Proof of Claim~\ref{lemma3.1}.}}
             It is equivalent to show that for any nonempty set \(I \subseteq [n]\),
             \(\sum_{i \in I} \mathbf{u_i} \neq 0\). Let \(I \subseteq [n]\). Note that $\mathbf{u_i}$ is a nonzero vector for all $1\le i\le n$. Then we can assume $|I|\ge 2$. Choose \(m \in I\) such that \(u_{m'} < u_m\)
              for any
             \(m' \in I \backslash \{m\}\).
             Since \(u_m < \{v_m,w_m\}\), then \(u_{m'} < \{v_m,w_m\}\) for any
             \(m' \in I \backslash \{m\}\).
             Note that either \(u_{m'} \rightarrow \{v_m,w_m\}\) or \(\{v_m,w_m\} \rightarrow u_{m'}\),
             which means \(\mathbf{u_{m'}} \cdot (\mathbf{v_m}+\mathbf{w_m})=0\).
             Since \(\mathbf{u_{m}} \cdot (\mathbf{v_m}+\mathbf{w_m})=1\), we have
             \((\sum_{i \in I} \mathbf{u_i}) \cdot (\mathbf{v_m}+\mathbf{w_m})=1\).
             Thus \(\sum_{i \in I} \mathbf{u_i} \neq 0\) as required.\q
\vskip.2cm
Let \(V(T)=\{x_1, x_2, \ldots, x_n\}\), where $x_i$ corresponds to $D_i$ for $1\le i\le n$.

         \begin{claim}\label{lemma3.2}
             For any nonempty set \(I \subseteq [n]\). Let \(S=\{\mathbf{u_i},\mathbf{v_i},\mathbf{w_i}|i \in I\}\).
             If \(\mathrm{rank}(S)=|I|\), then  \(T \left\langle \{x_i\}_{i \in I} \right\rangle\) is acyclic.
         \end{claim}

        \noindent \textbf{\emph{Proof of Claim~\ref{lemma3.2}.}}
         Let \(R=\{u_i,v_i,w_i| i \in I\}\).
             We prove it by induction on \(|I|\). Let \(Z= \{v_i,w_i|i \in I\}\) and choose \(z\in Z\) such that $z<y$ for any $y\in Z\setminus\{z\}$. Assume \(z \in \{v_t,w_t\}\), where \(t \in I\). Since \(\mathrm{rank}(S)=|I|\),
              there exists \(J \subseteq I\) such that \(\mathbf{z}= \sum_{j \in J} \mathbf{u_j}\) by
              Claim~\ref{lemma3.1}.

               If \(J = \emptyset\), then \(\mathbf{z} = 0\) which implies all arcs incident with \(z\)
              are not inverted. Hence  \(z\rightarrow y\) for all $y\in Z\setminus\{z\}$.
              So \(x_t \rightarrow x_i\) for each \(i \in I \setminus \{t\}\) in  \(T\) and \(z = v_t\) by \(v_t \rightarrow w_t\).
            Let \(R'=R\setminus \{u_t,v_t,w_t\}\) and
               \(S'=S\setminus \{\mathbf{u_t},\mathbf{v_t},\mathbf{w_t}\}\). Since \(u_t \rightarrow R'\) and \(u_t < R'\),
              we have \( \mathbf{u_t} \bot S'\) which implies \(\mathrm{rank}(S')=|I|-1\). By the inductive hypothesis, \(T \left\langle \{x_i\}_{i \in I \backslash \{t\}} \right\rangle\) is acyclic. Hence \(T \left\langle \{x_i\}_{i \in I } \right\rangle\) is acyclic
               as required.

              Now we consider the case  \(J\not = \emptyset\).
                If \(J=\{t\}\), then \(\mathbf{z}=\mathbf{u_t}\). Recall \(z \in \{v_t,w_t\}\). If \(z = v_t\) (resp. \(z = w_t\)), then \(\mathbf{v_t}=\mathbf{u_t}\) and \(v_t < w_t\) (resp. \(\mathbf{w_t}=\mathbf{u_t}\) and \(w_t < v_t\)) which implies
              \(1= \mathbf{u_t} \cdot \mathbf{w_t}=\mathbf{v_t} \cdot \mathbf{w_t}\) (resp. \(0= \mathbf{u_t} \cdot \mathbf{v_t}=\mathbf{w_t} \cdot \mathbf{v_t}\)),
             a
              contradiction with \(v_t \rightarrow w_t\). So we have \(J\not=\{t\}\). Let \(u_s\in \{u_j | j \in J\}\) such that $x<u_s$ for any $x\in \{u_j | j \in J\}\setminus\{u_s\}$. We consider two cases.

              \textbf{Case \(1\).} \(s \neq t\). Then
              \(\{z\} \cup \{u_j | j \in J\} < \{v_s , w_s\}\). For any $y\in \{z\} \cup \{u_j | j \in J\setminus\{s\}\}$, we have
               either
              \(y \rightarrow \{v_s , w_s\}\) or \(\{v_s , w_s\} \rightarrow y\) which implies
               \(\mathbf{y} \cdot (\mathbf{v_s}+\mathbf{w_s})=0\).
              Hence \((\mathbf{z}+\sum_{j \in J} \mathbf{u_j}) \cdot (\mathbf{v_s}+\mathbf{w_s})=1\)
              by \(\mathbf{u_s} \cdot (\mathbf{v_s}+\mathbf{w_s})=1\),
              a contradiction with \(\mathbf{z}= \sum_{j \in J} \mathbf{u_j}\).

              \textbf{Case \(2\).} \(s=t\). In this case, $u_s<z$ by assumption. Let \(u_{s'}\in
              \{u_j | j \in J\setminus \{t\}\}\) such that $x<u_{s'}$ for any $x\in \{u_j | j \in J\setminus \{t,s'\}\}$.
                            Then
              \(z < \{v_{s'} , w_{s'}\}\) and
              \( \{u_j | j \in J\} \le u_s < z < \{v_{s'} , w_{s'}\}\). For any $y\in \{z\} \cup \{u_j | j \in J\setminus\{s'\}\})$, we have
                either
              \(y \rightarrow \{v_{s'} , w_{s'}\}\) or \(\{v_{s'} , w_{s'}\} \rightarrow y\) which implies
              we have \(\mathbf{y} \cdot (\mathbf{v_{s'}}+\mathbf{w_{s'}})=0\).
              Therefore \((\mathbf{z}+\sum_{j \in J} \mathbf{u_j}) \cdot (\mathbf{v_{s'}}+\mathbf{w_{s'}})=1\)
             by \(\mathbf{u_{s'}} \cdot (\mathbf{v_{s'}}+\mathbf{w_{s'}})=1\),
              a contradiction with \(\mathbf{z}= \sum_{j \in J} \mathbf{u_j}\).
             \q

         Now we complete the proof of Theorem~\ref{thm1.5}. By  Claim~\ref{lemma3.1}, we have \(\mathrm{rank}(\Lambda) \ge n\) which implies
    \(k \ge n\). If $k=n$, then \(\mathrm{rank}(\Lambda) = n\). By  Claim~\ref{lemma3.2}, \(\invt=0\), our final contradiction.
   \hfill $\square$\par


     Since we can invert \(T\) to acyclic and
        then invert \(D_i\) one after another,
    then for oriented graphs \(D_i\) with \(\invdi=1\), the following inequality holds
    \[
         n \le \mathrm{inv}(T[D_1,D_2,\cdots,D_n]) \le n+\mathrm{inv}(T).
    \]
    Here we prove that there exists $T$ such that it cannot reach the upper bound.


\noindent \textbf{\emph{Proof of Theorem~\ref{example}.}}
From Theorem \ref{thm1.2}, for each odd integer $k\ge3$, there is a tournament $D$ with $\mathrm{inv}(D)=k$ such that  $\mathrm{inv}(D\Rightarrow\overrightarrow{C_3})\le k$. Since $D$ is a subgraph of $D\Rightarrow\overrightarrow{C_3}$, we have $\mathrm{inv}(D\Rightarrow\overrightarrow{C_3})=k$. From Theorem \ref{direction}, we know $\mathrm{inv}(D)=\mathrm{inv}(\overrightarrow{C_3}\Rightarrow D)=k$. Let \(T_k=u_1 \Rightarrow D\). It's not hard to find that $\mathrm{inv}(T_k)=k$, since $D$ is a subgraph of $T_k$ and the decycling family of $D$ gives a decycling family of $T_k$. Assume \(V(T_k)=\{u_1,u_2,\dots,u_{n_k}\}\). We need to show \(\mathrm{inv}(T_k[D_1,D_2,\dots,D_{n_k}]) \le n_k+k-1\) when $\mathrm{inv}(D_i)=1$ for all $i$.

From Lemma \ref{odd invertion}, there is a \(k\)-decycling family \((X_i)_{1 \le i \le k}\) of $D$ such that every characteristic vector is orthogonal to  \(\mathbf{1}\). It means $|\{i\in [k]|v\in X_i\}|$ is even for any $v\in V(D)$. The corresponding blow-up of $X_i$ is $X_i^{\prime}=\bigcup_{u_j \in X_i}V(D_j)$. Since $\mathrm{inv}(D_j)=1$, for each $j\in[n_k]$, there exists a $Y_j\subset V(D_j)$ such that Inv$(D_j;Y_j)$ is acyclic.
Then it is not hard to check that there exists a \((n_k+k-1)\)-decycling family as following:

        \begin{equation*}	
            \begin{aligned}
	  &Z_i = Y_1 \cup X_i^{\prime} &\quad 1 \le i \le k, \\
        &Z_{k+j-1} = Y_j &\quad 2 \le j \le n_k.
            \end{aligned}
        \end{equation*}
    Hence  \(\mathrm{inv}(T_k[D_1,D_2,\dots,D_{n_k}]) \le n_k+k-1\) and we are done.
   \hfill $\square$\par

\section*{Acknowledgement}
Y. Yang is supported by the Fundamental Research Funds for the Central University (Grant 500423306) in China. M. Lu is supported by the National Natural
Science Foundation of China (Grant 12171272 \& 12161141003).


\end{document}